\documentclass{amsart}

\title{Strichartz estimates on exterior polygonal domains}

\author[D. Baskin]{Dean Baskin}
\email{dbaskin@math.northwestern.edu}
\address{Mathematics Department, Northwestern University \\
Evanston, IL 60208, USA}
\author[J.L. Marzuola]{Jeremy L. Marzuola}
\email{marzuola@math.unc.edu}
\address{Mathematics Department, University of North Carolina \\
Phillips Hall, CB\#3250, Chapel Hill, NC 27599, USA}
\author[J. Wunsch]{Jared Wunsch}
\email{jwunsch@math.northwestern.edu}
\address{Mathematics Department, Northwestern University \\
Evanston, IL 60208, USA}

\renewcommand{\Im}{\operatorname{Im}}

\newcommand\bl{{\mathrm{b}}}

\newcommand\Diffb{\mathrm{Diff}_\bl}
\newcommand\Hb{H_\bl}
\newcommand{\NN}{\mathbb{N}}
\newcommand{\dom}{\mathcal{D}}
\newcommand{\Lap}{\Delta}
\newcommand{\tu}{\tilde{u}}
\newcommand{\tf}{\tilde{f}}
\newcommand{\tF}{\widetilde{F}}
\newcommand{\CI}{\mathcal{C}^\infty}
\newcommand{\jwnorm}[1]{{\left\|{#1}\right\|}}
\def\Vol{\text{Vol}}

\newcommand{\abs}[1]{{\left\lvert{#1}\right\rvert}}

\newcommand{\ang}[1]{{\left\langle{#1}\right\rangle}}
\newcommand{\smallabs}[1]{{\lvert{#1}\rvert}}

\newcommand{\norm}[2][]{\left\|#2\right\|_{#1}}

\newcommand{\pd}[1][]{\partial_{#1}}

\newcommand{\p}{\ensuremath{\partial}}
\newcommand{\lap}{\Delta}
\newcommand{\reals}{\mathbb{R}}
\newcommand{\RR}{\mathbb{R}}
\newcommand{\sphere}{\mathbb{S}}
\newcommand{\cone}{\ensuremath{C(\mathbb{S}^1_\rho)}}
\newcommand{\crc}{\mathbb{S}^1_\rho}

\newtheorem{theorem}{Theorem}
\newtheorem{lemma}{Lemma}

\newtheorem{conjecture}{Conjecture}
\theoremstyle{remark}

\newtheorem{remark}{Remark}

\thanks{The first author acknowledges the support of NSF postdoctoral
  fellowship DMS-1103436, and the third author was partially supported
  by NSF grant DMS-1001463.  The second author was partially supported
  by an IBM Junior Faculty Development Award through the University of
  North Carolina and would like to thank the University of Chicago for
  hosting him during the completion of part of this work.  The third
  author is grateful to Fabrice Planchon and John Stalker for
  a series of discussions of smoothing and Strichartz estimates on
  cones over a period of several years.  In addition, we thank Fabrice Planchon for a careful reading of the result and providing several references for the scattering theory discussion in the final section.}

\begin{document}

\begin{abstract}
  Using a new local smoothing estimate of the first and third authors,
  we prove local-in-time Strichartz and smoothing estimates without a
  loss exterior to a large class of polygonal obstacles with arbitrary
  boundary conditions and global-in-time Strichartz estimates without
  a loss exterior to a large class of polygonal obstacles with
  Dirichlet boundary conditions.  In addition, we prove a
  global-in-time local smoothing estimate in exterior wedge domains
  with Dirichlet boundary conditions and discuss some nonlinear
  applications.
\end{abstract}

\maketitle

\section{Introduction}

In this paper we prove a family of local- and global-in-time Strichartz estimates
for solutions to the Schr{\"o}dinger equation
\begin{equation}\label{eq:homSchro}
\left\{
\begin{split}
\left(D_t + \Delta\right) u(t,x) &= 0 \\
u(0,x) &= f(x),
\end{split}
\right.
\end{equation}
where $D_t=i^{-1} \partial_t,$ $\Delta$ is the negative definite
Laplace-Beltrami operator on domains of the form $X = \reals^2
\setminus P$ for $P$ any non-trapping polygonal region such that no
three vertices are collinear (as defined in the recent work of the
first and last author \cite{baskin:conedecay}), and where we take
either Dirichlet or Neumann boundary conditions for the Laplacian for the local result
and only Dirichlet boundary conditions for the global result.\footnote{The
  essential distinction between the Dirichlet and Neumann cases for our
  purposes is that low-energy resolvent estimates seem to be readily available in
  the literature in the Dirichlet case (e.g.\ in \cite{Burq-Acta-1998})  but not in the Neumann case.}
These assumptions and the resulting restrictions on allowed obstacles
$P$ are discussed in detail as Assumptions $1$, $2$ and $3$ in Section
$2$ of \cite{baskin:conedecay}, to which we refer the reader for more
details.  The main tools we require for the proof are the local
smoothing estimate on such domains (due to the first two authors) and
the Strichartz estimates on wedge domains (due to Ford
\cite{Ford:2010}).

We note here that to define the Laplacian, we use the standard Friedrichs extension, which
is the canonical self-adjoint extension of a non-negative densely defined symmetric
operator as defined in for instance \cite{Davies:1990,Blair:2012}. The Neumann Laplacian is taken
to be the usual Friedrichs extension of the Laplace operator acting on smooth functions
which vanish in a neighborhood of the vertices. The Dirichlet Laplacian is taken
to be the typical Friedrichs extension of the Laplace operator acting on smooth functions
which have compact support contained in $X$.

We now briefly discuss the geometric restrictions on $P$ needed to
apply the results of \cite{baskin:conedecay}.  In particular, we
review the non-trapping assumption on the exterior of $P$ as it is an
important restriction in all exterior domain results.  Let $P$ be a
polygonal domain in $\RR^2,$ not necessarily connected.  A
\emph{geometric geodesic} on $\RR^2\backslash P$ that does not pass
through the vertices of $P$ is defined as a continuous curve that is a
concatenation of maximally extended straight line segments in $\RR^2
\backslash P,$ such that on $\partial P,$ successive segments make
equal angles with the boundary (``specular reflection'').  More
generally, a geometric geodesic is one that may pass through the
vertices of $P$ in such a way that it is locally a uniform limit of
geometric geodesics missing the vertices.  This means that in general
such a geodesic has two possible continuations each time it hits a
vertex, corresponding to taking the limit of families approaching the
vertex from the left and right sides.

We let $B$ be a closed ball containing $P$ in
its interior.  We say that $P$ is \emph{non-trapping} if there exists
$T>0$ such that every geometric geodesic starting in $B$ leaves $B$ in time
less than $T$ (this condition is of course independent of the choice
of $B$).  We assume henceforth that $P$ is non-trapping.

In addition to assuming that $P$ is non-trapping, we also require the
assumption that no three vertices of $P$ are collinear along geometric geodesics.  We further
remark that the third assumption of \cite{baskin:conedecay}, requiring
that cone points be pairwise non-conjugate, is automatically satisfied
for the Euclidean domains under consideration here.

We recall that admissible Strichartz exponents for the Schr{\"o}dinger
equation in dimension $n=2$ are given by the following:
\begin{equation}
  \label{eq:admissible}
  \frac{2}{p} + \frac{2}{q} = 1, \quad p,q \geq 2, \quad (p,q)\neq (2,\infty).
\end{equation}
We are now ready to state the main result of this note.
\begin{theorem}
  \label{thm:strichartz}
  For any $(p,q)$ satisfying
  equation~\eqref{eq:admissible}, there is a constant $C_{p,q,T}$ so that on $I = [0,T]$
  \begin{equation*}
    \norm[L^{p}({I}, L^{q}(X))]{e^{-it\lap}f } \leq C_{p,q,T} \norm[L^{2}(X)]{f}
  \end{equation*}
  for all $u_{0}\in L^{2}(X)$.  If $X$ has Dirichlet boundary conditions, we can take $I = \RR$.
\end{theorem}

\begin{remark}
Using a now standard application of the Christ-Kiselev lemma \cite{Christ:2001}, we can conclude that for a solution $u$ to the inhomogeneous Schr\"odinger IBVP
\begin{equation}\label{eq:inhomSchro}
\left\{
\begin{split}
\left(D_t + \Delta\right) u(t,x) &= F(t,x) \\
u(0,x) &= f(x)
\end{split}
\right.
\end{equation}
satisfying either Dirichlet or Neumann homogeneous boundary
conditions, the estimate
\begin{equation}\label{eq:inhomStrich}
\|u\|_{L^{p_1}(I; L^{q_1}(X))} \leq C \left( \left\| f \right\|_{L^2(X)} + \left\| F \right\|_{L^{p_2'}(I;L^{q_2'}(X))} \right)  
\end{equation}
holds for $\frac{2}{p_j} + \frac{2}{q_j} = 1$ for $j = 1,2$.  Here, $(\cdot)'$ denotes the dual exponent, e.g.\ $\frac{1}{p_1} + \frac{1}{p_1'} = 1$.  
\end{remark}

\section{Global Strichartz Estimates for the Model Problems}
\label{s:model-strichartz}

The proof of the theorem relies on Strichartz estimates on
$\reals^{2}$, as well as Strichartz estimates on a two-dimensional
cone, which we recall the result here for completeness.  Here and in
what follows, we denote by $C(\sphere^{1}_{\rho})$ the cone over the
circle of circumference $\rho$, equipped with the conic metric $dr^{2}
+ r^{2}\,d\theta^{2}$.

\begin{theorem}[Strichartz estimates on $\reals^{2}$, $C(\sphere^{1}_{\rho})$; see, e.g.,
  Keel--Tao~\cite{Keel:1998} and Ford~\cite{Ford:2010}]
  \label{thm:strichartz-R2-C}
  Suppose that $(p,q)$ and $(\tilde{p},\tilde{q})$ are admissible
  Strichartz exponents in the sense of equation~\eqref{eq:admissible}.
  If $u$ is a solution to the Schr{\"o}dinger equation
\begin{equation*}
\left\{
\begin{split}
\left(D_t + \Delta_Y \right) u(t,x) &= F(t,x), \\
u(0,x) &= f(x),
\end{split}
\right.
\end{equation*}
  on $Y = \reals^{2}$ or $C(\sphere^{1}_{\rho})$, then
  \begin{equation*}
    \norm[L^{p}({\reals}; L^{q}(Y))]{u} \leq C
    (  \norm[L^{2}]{f} + \| F \|_{L^{\tilde{p}'} (\reals; L^{\tilde{q}'} (Y)}),
  \end{equation*}
  where $\tilde{p}'$ and $\tilde{q}'$ are the conjugate exponents to
  $p$ and $q$, respectively.
\end{theorem}

\begin{remark}
Our results are closely related to the work on smoothing and
Strichartz estimates for potentials with inverse-square singularities by
  Burq--Planchon--Stalker--Tahvildar-Zadeh and
  Planchon--Stalker--Tahvildar-Zadeh in \cite{BurqPlanchonST-Z:2003} and
  \cite{PlanchonST-Z:2003}.  For work on smoothing estimates for
  the Schr\"odinger equation in \emph{smooth} exterior domains, we refer the reader
  to the early works of Burq \cite{Burq:2004a}, as well as Burq--Gerard--Tzvetkov \cite{BurqGT:2004} and Anton \cite{Anton:2008} who constructed parametrices for exterior domain problems that proved Strichartz estimates with errors controlled by local smoothing estimates.  Local smoothing results were later extended by Robbiano--Zuily
  \cite{Robbiano:2009} to include quadratic potential wells.  Scale invariant Strichartz estimates for exterior domains first appeared in Planchon--Vega \cite{PlanchonVega:2009} and Blair--Smith--Sogge \cite{Blair:2010}, though not for the full range of admissible Strichartz pairs.  For Strichartz estimates exterior to a smooth, convex obstacle however, scale invariant estimates have been established in the full range of estimates in
Ivanovici \cite{Ivanovici:2010}, Ivanovici--Planchon
\cite{Ivanovici:2010a},  and
Blair \cite{Blair:2011a}.
\end{remark}

\begin{remark}
  Strichartz estimates exist for the wave equation on both $\reals^2$
  and $\cone$, but the analog of Theorem \ref{thm:strichartz} for the
  wave equation on exterior domains can be directly computed from the
  analysis done by Blair, Ford and the second author in
  \cite{Blair:2011} due to the finite speed of propagation.  Hence,
  quantifying the effects of diffraction as in \cite{baskin:conedecay}
  plays a much larger role in Schr\"odinger dynamics than in the
  corresponding wave dynamics on such domains.
\end{remark}

\section{Local Smoothing Estimates for $X$ and on Euclidean Cones}
\label{section:locsmoothing}

The proof of Theorem \ref{thm:strichartz} will also rely upon local
smoothing estimates for $\reals^2$, $C(\sphere^{1}_{\rho})$ as well as on
the space $X$ in order to glue together similar dispersive results on model
problems.  We begin with the local result that is independent of choice of
boundary conditions for $X$ as a consequence of being local-in-time (and
thus requiring no low-energy resolvent estimates):

\begin{theorem}[``Local'' local smoothing estimate; see B.--W.~\cite{baskin:conedecay}]
  \label{thm:local-smoothing_local}
  If $X$ is a domain exterior to a non-trapping polygon, $u$ is a
  solution of the Schr{\"o}dinger equation
\begin{eqnarray*}
\left\{   \begin{array}{c}
 D_t u(t,z) + \lap u(t,z) = 0, \\
    u(0,z) = f(z),
  \end{array} \right.
  \end{eqnarray*}
with Dirichlet or Neumann boundary conditions
  and $\chi \in C^{\infty}(X)$ is a smooth, compactly supported
  function, then $u$ satisfies a local smoothing estimate
  \begin{equation*}
    \norm[L^{2}({[0,T]}; \mathcal{D}_{1/2})]{\chi u} \leq C_{T} \norm[L^{2}]{f},
  \end{equation*}
  where $\mathcal{D}_{1/2}$ is the domain of $(-\lap)^{1/4}$.  
\end{theorem}


In the case of Dirichlet boundary conditions, we can significantly strengthen the above result to gain global control
over the local smoothing norm.

\begin{theorem}[``Global'' local smoothing estimate]
  \label{thm:local-smoothing_global}
  If $X$ is a domain exterior to a non-trapping polygon, $u$ is a
  solution of the Schr{\"o}dinger equation
\begin{eqnarray*}
\left\{   \begin{array}{c}
 D_t u(t,z) + \lap u(t,z) = 0, \\
    u(0,z) = f(z),
  \end{array} \right.
  \end{eqnarray*}
with Dirichlet boundary conditions
  and $\chi \in C^{\infty}(X)$ is a smooth, compactly supported
  function, then $u$ satisfies a local smoothing estimate
  \begin{equation*}
    \norm[L^{2}(\reals; \dom_{1/2})]{\chi u} \leq C \norm[L^{2}]{f},
  \end{equation*}
  where $\dom_{1/2}$ is the domain of $(-\lap)^{1/4}$.  
\end{theorem}

\begin{remark}
  Note that Theorems~\ref{thm:local-smoothing_local}, \ref{thm:local-smoothing_global} imply the dual
  estimate
  \begin{equation}\label{dualsmoothing}
    \norm[L^{2}]{\int _{I}e^{is\lap}\chi F(s) \,ds} \leq
    C \norm[L^{2}({\RR}; \dom_{-1/2})]{F},
  \end{equation}
  for $I$ either $[0,T]$ or $\RR$ respectively.
\end{remark}

\begin{proof}
  We rely on the high-frequency resolvent estimates of the first and
  third authors~\cite{baskin:conedecay}, estimates due to
  Morawetz~\cite{Morawetz:1975} for intermediate frequencies and Burq \cite{Burq-Acta-1998} for small frequencies, then apply a now-standard $TT^{*}$ argument.

  Consider the operator $Tu_{0} = \chi e^{-it\lap}u_{0}$.  We wish to show
  that $T$ is a bounded operator from $L^{2}(X)$ to $L^{2}(\reals ;
  \dom_{1/2})$.  It suffices to show that $TT^{*}$ is bounded
  from $L^{2}(\reals ; \dom_{-1/2})$ to $L^{2}(\reals;
  \dom_{1/2})$.  The operator $TT^{*}$ is given by
  \begin{align*}
    TT^{*}f &= \chi\int_{\reals}e^{-i(t-s)\lap}\chi f(s)\,ds \\
    &= \chi\int_{s < t}e^{-i(t-s)\lap}\chi f(s)\,ds + \chi \int_{s >
      t}e^{-i(t-s)\lap}\chi f(s)\,ds = \chi T_{+}f + \chi T_{-}f.
  \end{align*}
  
  Observe that $T_{\pm}f$ are both solutions of the inhomogeneous
  Schr{\"o}dinger equation
  \begin{align*}
    D_{t} u + \lap u = \frac{1}{i}\chi f.
  \end{align*}
  Suppose for now that $f$ is compactly supported in time, i.e.,
  $f(t,x) = 0$ for $t \notin [-t_{0},t_{0}]$.  In this case, $T_{+}f$
  vanishes for $t < -t_{0}$ and $T_{-}f$ vansishes for $t >
  t_{0}$. 
  
  We wish to show that there is a constant $C$, independent of
  $t_{0}$, so that 
  \begin{equation*}
    \int _{\reals}\norm[\dom_{1/2}]{\chi T_{\pm}f(t,x)}^{2} \,dt
    \leq C \int_{\reals}\norm[\dom_{-1/2}]{f(t,x)}^{2}\,dt.
  \end{equation*}
  By Plancherel's theorem, it suffices to show that
  \begin{equation*}
    \int_{\reals}\norm[\dom_{1/2}]{\chi
      \widehat{T_{\pm}f}(E, x)}^{2}\,dE \leq C
    \int_{\reals}\norm[\dom_{-1/2}]{\hat{f}(E,x)}^{2}\,dE, 
  \end{equation*}
  where $\hat{f}$ denotes the Fourier transform of $f$ in $t$.
  
  Observe that $\widehat{T_{\pm}f}(E,x)$ solve
  \begin{equation*}
    (\lap + E) \widehat{T_{\pm}f} = \frac{1}{i}\chi \hat{f}.
  \end{equation*}
  Moreover, the condition on the support of $f$ implies that
  $\widehat{T_{+}f}$ is holomorphic in the lower half-plane, while
  $\widehat{T_{-}f}$ is holomorphic in the upper half-plane.  In
  particular, if $R(z) = (\lap + z)^{-1}$ where it is invertible,
  \begin{equation*}
    \widehat{T_{\pm}f}(E,x) = \lim _{\beta \downarrow 0} R(E \mp
    i\beta) \left( \frac{1}{i}\chi \hat{f}(E,x)\right).
  \end{equation*}

  We must thus estimate $\chi R(E\mp i 0) \chi$ as an operator
  $\dom_{-1/2}\to \dom_{1/2}$.  The high-frequency
  estimates from~\cite{baskin:conedecay} imply that there is some $E_{0}$ so
  that for $E \geq E_{0}$,
  \begin{equation*}
    \norm[L^{2}\to L^{2}]{\chi R(E\mp i0)\chi} \leq \frac{C}{\sqrt{E}}.
  \end{equation*}
  Using this bound and the identity
  \begin{equation*}
    \lap \chi R(E\pm i\beta) = \chi - (E\pm i\beta)\chi R(E\pm i\beta)
    + [\lap, \chi]R(E\pm i\beta)
  \end{equation*}
  yields the following high-energy estimate for the resolvent:
  \begin{equation*}
    \norm[L^{2}\to \dom_{2}]{\chi R(E\pm i0)\chi} \leq C\sqrt{E}.
  \end{equation*}
  Interpolating the two estimates shows that $R(E\pm i0)$ is bounded
  (with uniform bound for $E \geq E_{0}$) as an operator from $L^{2}$
  to $\dom_{1}$ and thus from $\dom_{-1/2}$ to
  $\dom_{1/2}$.

  The argument of Morawetz~\cite[Lemmas 15 and 16]{Morawetz:1975},
  which remains valid in our setting, shows that the same bound holds
  at intermediate energies as well.  For uniform bounds down to $E = 0$, 
  we rely upon an argument of Burq~\cite[Appendix B.2]{Burq-Acta-1998}.

  We now apply the resolvent estimates, which shows that there is a
  constant $C$ independent of $\lambda$ and $t_{0}$ so that
  \begin{equation*}
    \norm[\dom_{1/2}]{\widehat{\chi T_{\pm}f(E,x)}}^{2} \leq C \norm[\dom_{-1/2}]{f}^{2}.
  \end{equation*}
  Integrating in $E$ then finishes the proof in the compactly
  supported setting.  For the general setting, we simply note that the
  constant is independent of the support and that compactly supported
  functions are dense in $L^{2}_{t}$.
\end{proof}

We will need one result that we have not been able to find explicitly
in the literature, but whose proof uses standard methods.  This result
concerns global-in-time local smoothing the Schr\"odinger equation on an infinite
wedge domain, which of course serves as a local model for our polygon
near a vertex and is equivalent to $C(\sphere^{1}_{\rho})$ (Cf.\
\cite{HHM}, for instance).  Let $X_\rho=\{\theta \in [0, \rho/2]\}
\subset \RR^2$ for $\rho\in [0, 4\pi).$

\begin{lemma}\label{lemma:wedgesmoothing_global}
A solution to the Schr\"odinger equation on $X_\rho$ with Dirichlet or
Neumann conditions satisfies the local smoothing estimates of
Theorem~\ref{thm:local-smoothing_global}.  Consequently, the dual estimate
\eqref{dualsmoothing} is satisfied on $X_\rho$ as well.
\end{lemma}

\begin{proof}
We first note that
solutions to the Schr\"odinger equation on $X_\rho$ with Dirichlet boundary conditions are equivalent, by extending in an odd manner to the cone over the circle of circumference $\rho$ obtained by
``doubling'' the wedge $X_\rho .$  That is to say, we may identify solutions
to the Schr\"odinger equation on $\RR\times X_\rho$ to 
solutions on the ``edge manifold'' $\RR\times \cone,$ where as usual $\cone$ has
the metric $ds^2 = dr^2 + r^2 d\theta ^2$
with $\theta \in \mathbb{S}^1_\rho.$  We make this identification by
extending the the solutions to be odd or even
under the involution
$$
\mathbb{S}^1_\rho\ni \theta \to \rho-\theta
$$
according to the choice of boundary condition.
Thus, it will suffice to consider solutions to the Schr\"odinger equation
on the cone,
\begin{equation}\label{eq:coneSchropol}
\left\{
\begin{split}
\big(D_t + \partial_r^2 + r^{-1} \partial_r  & + r^{-2} \Delta_{\crc} \big) u(t,x) = 0 , \\
u(0,x) &= f(x).
\end{split}
\right.
\end{equation}
We refer the reader to \cite{Ford:2010,Blair:2012,Blair:2011} as well
as \cite{Melrose:2004}
for a discussion of Sobolev spaces on cones and the
nature of the operator $\Delta_{\cone}$.   In particular, we briefly
recall the characterization given in \cite{Melrose:2004} of the nature
of the domains of powers of the Friedrichs Laplacian on the
cone.\footnote{As discussed in Remark 1.2\cite{Blair:2012} we take the
  Neumann Laplacian on the original planar domain to be the Friedrichs
  extension from smooth functions vanishing at the vertex or vertices,
  satisfying Neumann conditions at edge; thus, upon doubling to a cone, we
  are working with the Friedrichs extension of the Laplacian on smooth
  functions compactly supported away from the cone tip.}
First we recall the definition of b-vector fields and operators.
The space of \emph{b-vector fields}, denoted $\mathcal{V}_{b}(\cone)$ is the
vector space of vector fields on $[0,\infty)\times\mathbb{S}^1_\rho$
tangent to $0 \times \mathbb{S}^1_\rho$.  In local coordinates $(r,\theta)$ near $\pd M$,
they are generated over $C^{\infty}([0,\infty) \times \mathbb{S}^1_\rho)$ by the vector fields
$r\partial_r$ and $\partial_\theta$.  One easily verifies that $\mathcal{V}_{b}(\cone)$ forms a Lie
algebra.  The set of b-differential operators, $\Diffb^{*}(\cone)$, is the
universal enveloping algebra of this Lie algebra:
it is the filtered algebra consisting of operators of the form
\begin{equation}\label{exampleboperator}
A=\sum_{\smallabs{\alpha}+j\leq m} a_{j,\alpha}(r,\theta) (r D_r)^j
D_\theta^\alpha \in \Diffb^m(\cone).
\end{equation}
Now let $L^2_b(\cone)$ be the space of square-integrable functions
with respect to the ``b-density'' $r^{-1} \, dr\, d\theta.$  We define
the b-Sobolev spaces $\Hb^m(\cone)$ for $m \in \NN$ as
$$
\big \{u : \Diffb^m(\cone): u \to L^2_b(\cone)\big\}.
$$
This definition can be extended to a definition of $\Hb^s(\cone)$ for
$s \in \RR$ by interpolation and duality, or, better yet, by
developing the b-pseudodifferential calculus as in \cite{Melrose:APS}.

Finally, we recall that Lemma~3.2 of \cite{Melrose:2004} tells us
that if we let $\dom_s$ denote the domain of $\Lap_{\cone}^{s/2}$ (again, with
$\Lap_{\cone}$ denoting the Friedrichs Laplacian) then
$$
\dom_s=r^{-1+s}\Hb^s(\cone),\quad \abs{s}<1.
$$
This identification does break down at $s=1$---see \S3 of
\cite{Melrose:APS} for details.

Finally we are ready to prove a local smoothing estimate.
We will prove an estimate of the form
\begin{equation}
  \label{eq:conesmoothing-redux}
  \int_{0}^{\infty}\left( \norm[L^{2}]{\langle r\rangle
      ^{-\frac{3}{2}} \pd[r] u}^{2} +
    \norm[L^{2}]{r^{-\frac{3}{2}}\sqrt{-\lap_{\sphere^{1}_{\rho}}}
      u}^{2}\right) \,dt \leq C\norm[\dom_{1/2}]{f}^{2} .
\end{equation}
Taken together with its time-reversed version, this will yield the
estimate
$$
\jwnorm{\chi u}_{L^2(\RR; \dom_{1})} \leq C \jwnorm{f}_{\dom_{1/2}}
$$
(indeed it is somewhat stronger than the needed estimate, being global
in space, with weights, rather than local).  Now to obtain the
estimate stated in the theorem, we simply shift Sobolev exponents by
applying this estimate to the solution $\ang{\Lap}^{-1/4} u.$  Thus to
prove the lemma it
will suffice to obtain \eqref{eq:conesmoothing-redux}.

To do this, we separate variables and treat the small and large angular
frequencies separately.

Let us first introduce the commutants $A_{0}$ and $B_{0}$, given by
\begin{equation*}
  A_{0}= \pd[r], \quad B_{0} = \frac{r}{\langle r \rangle} \pd[r]
\end{equation*}
and observe that acting on smooth functions compactly supported away
from the cone tip,
\begin{align*}
  [A_{0}, -\lap_{C(\sphere^{1}_{\rho})}] &= \frac{1}{r^{2}}\pd[r] +
  \frac{2}{r^{3}}\pd[\theta]^{2}, \\
  [B_{0}, -\lap_{C(\sphere^{1}_{\rho})}] &= 2\left( \frac{1}{\langle
      r\rangle^{3}}\pd[r]^{2} + \frac{1}{r^{2}\langle
      r\rangle}\pd[\theta]^{2}\right) + \frac{1}{r\langle r
    \rangle}\left( 1 + \frac{1}{\langle r\rangle^{2}} -
    \frac{3r^{2}}{\langle r\rangle^{4}}\right)\pd[r].
\end{align*}
Note that the formal adjoints (with respect to the usual volume form)
are given by
\begin{align*}
  A_{0}^{*} &= -\pd[r] - \frac{1}{r}, \\
  B_{0}^{*} &= -\frac{r}{\langle r \rangle}\pd[r] - \frac{2 +
    r^{2}}{\langle r \rangle ^{3}}.
\end{align*}
Now setting $A = (A_{0}-A_{0}^{*})/2$ and $B = (B_{0} - B_{0}^{*})/2$,
we have
\begin{equation}\label{commutators}
\begin{aligned}
  \left[ A, - \lap_{C(\sphere^{1}_{\rho})}\right] &=
  \frac{2}{r^{3}}\pd[\theta]^{2} + \frac{1}{2r^{3}}, \\
  \left[ B, -\lap_{C(\sphere^{1}_{\rho})}\right] &= -2\pd[r]^{*}
  \langle r \rangle ^{-3} \pd[r] + 2\langle r\rangle
  ^{-1}r^{-2}\pd[\theta]^{2} + g,
\end{aligned}
\end{equation}
with
\begin{equation*}
g(r) = \frac{r^{4}+8r^{2}-8}{2\langle r \rangle ^{7}}.
\end{equation*}

Now consider a function $u \in \CI(\RR; \CI_c((0,\infty)\times
\mathbb{S}^1_\rho)$ with
\begin{equation}\label{inhomog}
(D_t+\Lap_{\cone})u=F,\ u(0,x)=f.
\end{equation}

We will separate $u(t)$ in angular modes, preserving a
single high-angular-frequency component and separating out all low-angular-frequency
components: given $J,$ we let $e_j(\theta)=e^{2\pi i j \theta/\rho}$
and we Fourier analyze $u$ into
$$
u=\sum_j a_j (t,r) e_j(\theta)\equiv \sum u_j(t,r,\theta).
$$
We then split
$$
u= \tu + \sum_{\abs{j}<J} u_j(t,r)
$$
with
$$
\tu= \sum_{\abs{j}\geq J} a_j (t,r) e_j(\theta)
$$
denoting the high angular frequencies.  Let $\tF,$ $\tf$ etc.\ denote
the corresponding decompositions of $F,$ $f.$

Now for $j$ sufficiently large,
$$
\ang{-\partial_\theta^2 u_j,u_j} \gg 1,
$$
and hence by \eqref{commutators}, there exists $c>0$ with
\begin{equation}\label{highfreqbound}
 \ang{[B, \Delta_{\cone}] \tu,\tu}  \geq c\big(  \norm[L^{2}]{\langle r\rangle
      ^{-\frac{3}{2}} \pd[r] u}^{2} +
    \norm[L^{2}]{r^{-\frac{3}{2}}\sqrt{-\lap_{\sphere^{1}_{\rho}}}
      u}^{2}\big).
\end{equation}
Thus, we compute
$$
i^{-1} \partial_t\ang{B\tu,\tu} = \ang{[B,-\Delta_{\cone}]\tu,\tu}+2 \Im \ang{\tF,B\tu}.
$$
Integration then yields
\begin{align*}
 \int_0^T \langle [B, \Delta_{\cone}] \tu, \tu \rangle dt
& \leq
  \abs{\ang{B\tu(T),\tu(T)}} +\abs{\ang{B\tu(0),\tu(0)}} \\
  & \hspace{1.0cm} + 2\int_0^T
  \abs{\ang{\tF,B\tu}}\, dt.
\end{align*}
Since $B \in \Diffb^1(\cone),$ we certainly have
$$
B: \dom_{1/2}\to \dom_{-1/2}=\dom_{1/2}^*,
$$
by the identification $\dom_{\pm 1/2}=r^{-1\pm 1/2} \Hb^{\pm 1/2}(\cone).$
Thus, for all $T,$
\begin{align*}
\int_0^T \big(  \norm[L^{2}]{\langle r\rangle
      ^{-\frac{3}{2}} \pd[r] u}^{2} &+
    \norm[L^{2}]{r^{-\frac{3}{2}}\sqrt{-\lap_{\sphere^{1}_{\rho}}}
      u}^{2}\big) \, dt \leq  \\
      & c \int_0^T \jwnorm{u}_{\dom_{1/2}}
    \jwnorm{F}_{\dom_{1/2}}\, dt+\jwnorm{f}_{\dom_{1/2}}^2
+ \jwnorm{u(T)}_{\dom_{1/2}}^2.
\end{align*}
An elementary energy estimate for the inhomogeneous equation shows
that since $e^{-it\Lap}: \dom_s\to \dom_s$ for all $s,$ for all $t>0$
we have
$$
 \jwnorm{u(t)}_{\dom_{1/2}}\leq  \jwnorm{f}_{\dom_{1/2}}+\int_0^\infty
 \jwnorm{F(s)}_{\dom_{1/2}} \, ds.
$$
Thus for all $T>0,$
\begin{equation}\label{commutatorestimate1}
\int_0^T \big(  \norm[L^{2}]{\langle r\rangle^{-\frac{3}{2}} \pd[r] u}^{2} +
    \norm[L^{2}]{r^{-\frac{3}{2}}\sqrt{-\lap_{\sphere^{1}_{\rho}}}
      u}^{2}\big) \, dt \leq c\big(\int_0^\infty
    \jwnorm{F}_{\dom_{1/2}}\big)^2 + c\jwnorm{f}_{\dom_{1/2}}^2,
\end{equation}
with the constant $c$ independent of $T.$
Thus, the map
$$
(f,F) \in \dom_{1/2}\oplus L^1(\RR; \dom_{1/2}) \to u
$$
extends by continuity to yield \eqref{commutatorestimate1} in
particular\footnote{The reader may note that in obtaining these estimates
  we have not pursued optimality in $F:$ the solution $u$ should be one
  derivative more regular.  We have avoided this issue owing to the
  breakdown of the identification of the domain $\dom_1$ with a weighted
  b-Sobolev space (principally relevant in our analysis of the
  zero-angular-mode below).  We of course obtain the correct mapping
  properties of the inhomogeneous Schr\"odinger equation from $L^2
  \dom_{-1/2}$ to $L^2 \dom_{1/2}$ ex post facto by duality.}  with any $f
\in \dom_{1/2}$ and $F=0.$ Letting $T\to \infty,$ this is the desired estimate for the $\tu$ term.

We now turn to the $u_j$ terms, for which we will need the commutant
$A$ as well.  We treat the cases $j \neq 0$ and $j=0$ separately.  For
$j \neq 0$ we note that $\ang{\partial_\theta^2 e_j,e_j}<-1/4$ owing
to our assumption that $\rho<4\pi.$  Thus
$$
\ang{[A, -\Delta_{\cone}] u_j,u_j} \geq c (\jwnorm{r^{-3/2} \partial_\theta
  u_j}^2_{L^2}+\jwnorm{r^{-3/2} u_j}^2_{L^2}).
$$
Consequently, using the commutant $B+\epsilon^{-1} A$ in the argument
above with $\epsilon$ sufficiently small formally yields the desired estimate on
$u_j,$ for $j \neq 0.$  Note that it is essential that $A \in r^{-1}
\Diffb^1(\cone),$ hence $A: \dom_{1/2}\to \dom_{-1/2}=\dom_{1/2}^*$
so that we may proceed just as above.

Finally, we deal with $j=0.$  Here we employ the commutant $B-\epsilon^{-1} A,$ where we
remark that on $u_0,$ a crucial sign flips and we have
$$
\ang{[A,- \Delta_{\cone}] u_0,u_0} \leq -c (\jwnorm{r^{-3/2} \partial_\theta
  u_0}^2_{L^2} +\jwnorm{r^{-3/2} u_0}^2_{L^2})
$$
(where of course we really have $\partial_\theta u_0=0$).
\end{proof}

\begin{remark}
  The above proof, while appealingly simple, does not extend to the case of a slit
  obstacle, i.e.\ $\rho = 4 \pi$ or indeed to any product cone in which
  $-1/4$ is in the spectrum of the Laplacian on the link.  However, addition of a third operator,
\[
W_0 = f( r) \partial_r ,
\]
with $f$ to be determined allows these cases to be handled in the same fashion as above.  
We calculate in general:
\[
W_0^*  = - W_0 - \frac{f( r)}{r} - f' ( r),
\]
giving the multiplier
$$
W  = \frac{ W_0 - W_0^*}{2} = f( r) \p_r + \frac{f(  r)}{2r} + \frac{f' ( r)}{2}.
$$
Thus,
\begin{align*}
\ang{[W, -\Delta_{\cone}] u,u} 
& = - \ang{ 2  f' ( r) \p_r u, \p_r u} + \ang{\frac{ f'' ( r)}{r} u, u}  +   \ang{\frac{ f''' ( r)}{2} u, u} \\
& \hspace{1.0cm} -  \ang{\frac{ f' ( r)}{2 r^2} u, u} .
\end{align*}
Consequently, provided
\begin{itemize}
\item $f \in C^3$ with uniformly bounded derivatives, 

\item $f' < 0$,

\item $f'' (r) / r + f''' (r) /2 - f' (r) /( 2 r^2) > 0$
\end{itemize}
we obtain an estimate.  
In particular, taking
$ f(s) = (1 + r)^{-\frac12}$ 
gives a positive operator satisfying bounds as in \eqref{eq:conesmoothing-redux} with slightly different weights,
\begin{equation}
  \label{eq:conesmoothing-redux-slit}
  \int_{0}^{\infty}\left( \norm[L^{2}]{ (1+r)^{-\frac34} \pd[r] u}^{2} +
    \norm[L^{2}]{r^{-1} \langle r \rangle^{-\frac{3}{4}}\sqrt{-\lap_{\sphere^{1}_{\rho}}}
      u}^{2}\right) \,dt \leq C\norm[\dom_{1/2}]{f}^{2} 
\end{equation}
on this mode.  Note, we have made no attempt to optimize weights here.
\end{remark}

\section{Proof of Theorem~\ref{thm:strichartz}}

We are now ready to prove the main result, which follows from the same arguments whether $I = [0,T]$ or $\RR$.  Suppose that the polygon is contained in a ball of radius $R_{0}$ in
$Z_{0} = \reals^{2}$ and let $U_{0} = \reals^{2}\setminus
\overline{B(0,R_{0})}$.  For each vertex of the polygon, we let
$U_{j}$ be a neighborhood of that vertex in $X$ so that $U_{j}$ can
be considered as a neighborhood of the cone point in a wedge domain
$Z_{j}$ given by $\{\theta \in [0, \rho/2]\} \subset \RR^2.$  We may assume that the union of the $U_{j}$ covers
$X\setminus U_{0}$.  Let $\chi_{0}, \chi_{1}, \ldots , \chi_{N}$ be a
partition of unity subordinate to this cover of $X$.

Set $u$ to be the solution of the Schr{\"o}dinger equation with
initial data $f$, i.e.,
\begin{equation*}
  \left\{
    \begin{split}
      \left(D_t + \lap \right) u &= 0, \\
      u(0,z) &= f(z).
    \end{split}
  \right.
\end{equation*}
Consider now $u_{j}=\chi_{j}u_{j}$.  Note that $u_{j}$ solves the
following inhomogeneous Schr{\"o}dinger equation on $Z_{j}$:
\begin{equation*}
  \left\{
    \begin{split}
      D_{t} u_{j}+ \lap u_{j} &= [\lap, \chi_{j}]u , \\
      u_{j}(0,z) &= \chi_{j}(z)f(z).
    \end{split}
  \right.
\end{equation*}
We write $u_{j} = u_{j}' + u_{j}''$, where $u_{j}'$ is the solution
of the homogeneous equation on $Z_{j}$ with the same initial data
and $u_{j}''$ is the solution of the inhomogeneous equation with
zero initial data.  We know by \cite{Keel:1998} (for $Z_{0}$) and by
\cite{Ford:2010} (for $Z_{j}$, $j > 0$) that $u_{j}$ satisfies the
homogeneous Strichartz estimate.  

We now set $v_{j}(t,z) = [\lap, \chi_{j}]u.$  Then by Duhamel's Principle,
\begin{equation*}
  u_{j}'' = \int_{0}^{t}e^{-i(t-s)\lap_{Z_{j}}} v_{j}(s) \, ds. 
\end{equation*}
Note that $[\lap, \chi_{j}]$ is a compactly supported differential
operator of order $1$ supported away from the vertices and so the
local smoothing estimate for $X$ implies that there is a constant
$C$ so that
\begin{equation*}
  \norm[L^{2}({\RR}, \dom_{-1/2})]{v_{j}} \leq C\norm[L^{2}(X)]{f}.
\end{equation*}
We wish to show that $u_{j}''$ obeys the Strichartz estimates.  As
we are assuming that $p > 2$, the Christ--Kiselev
lemma~\cite{Christ:2001} implies that it is enough to show the
estimate for
\begin{equation*}
  \int_{\RR}e^{-i(t-s)\lap_{Z_{j}}}v_{j}(s) \, ds =
  e^{-it\lap_{Z_{j}}} \int_{\RR}e^{is\lap_{Z_{j}}}v_{j}(s) \, ds.
\end{equation*}
By the dual local smoothing estimate for $Z_{j}$ from Lemma~\ref{lemma:wedgesmoothing_global},
\begin{equation*}
  \norm[L^{2}]{\int_{\RR}e^{is\lap_{Z_{j}}} v_{j}(s)} \leq
  C\norm[L^{2}({\RR};\dom_{-1/2})]{v_{j}} \leq C\norm[L^{2}]{f}.
\end{equation*}
Applying the homogeneous Strichartz estimate for the propagator then
finishes the proof.

\section{A brief comment on nonlinear applications}

Let $X = \reals^2 \setminus P$, for $P$ a non-trapping polygonal domain with
either Dirichlet or Neumann boundary conditions on each edge.  We use our
loss-less local in time Strichartz estimates to extend well-posedness
results to problems of the form
\begin{eqnarray}
\label{e:nls}
D_t u + \lap u = \beta (|u|^2 ) u, \ \ u_0 \in H^s (X)
\end{eqnarray}
for some $s \geq 0$, where $\beta$ is any polynomial such that $\beta (0) = 0.$ For a more
general set of assumptions on $\beta$, see for instance the
treatments in the books of Cazenave \cite{Cazenave} and Tao \cite{Tao:2006a}
and the references therein, as well as a succinct exposition in the
recent survey article of D'Ancona \cite{DAncona}.  Largely, what follows will mirror the discussions in $\reals^2$ from \cite{Cazenave}.  The arguments throughout implicitly rely upon
Sobolev embeddings and Gagliardo-Nirenberg inequalities extending from $\reals^2$ to $X$.  
Note, the nonlinear Schr\"odinger equation \eqref{e:nls} has conservation of mass
\begin{align}
\label{e:mass}
M( u) & = \| u \|_{L^2 (X) } \\
& \hspace{.25cm} = M(u_0) 
\end{align}
and conservation of energy
\begin{align}
\label{e:energy}
E( u) & =  \| \nabla u \|_{L^2 (X)} + \int_X B(u^2) d \Vol_X  \\
& \hspace{1cm} = E(u_0), \notag  
\end{align}
where
\[
B (z) =  \int_0^z \beta(y) dy.
\]
These conservation laws that are quite useful for the study of well-posedness in that they allow one to control the $L^2$ and $H^1$ norms of the solution for a variety of nonlinearities $\beta$.

\subsection{Local Nonlinear Results}
The main results are of the following form.  

\begin{theorem}
\label{t:locwp}
Let $0 \leq \sigma < 1$.  Given $\| u_0 \|_{H^\sigma} < M$, there exists a $T_{max}(M)$, $T_{min} (M)
\leq \infty$ such that \eqref{e:nls} has a unique solution $u \in C(
(-T_{min},T_{max}), H^s) \cap L^q_{(-T_{min},T_{max})} L^r (X)$ with
continuous dependence upon the initial data.  In addition, if $T_{max}
\  (T_{min} ) \ < \infty$, then $\| u \|_{H^s} \to \infty$ as $t \to
T_{max}$ ($T_{min}$).  In particular, for $L^{2}$-subcritical
nonlinearities, $T_{\max} =T_{\min} = \infty$.
\end{theorem}

\begin{proof}

We discuss the proof in some special cases, citing proper references for further details. 

Given the conservation laws, if $\sigma = 1$, the results follow from
careful estimates using the density of smooth functions in $H^1 (X)$ (see
Theorem~3.3.5 in \cite{Cazenave}).  Hence, the primary contribution to
well-posedness theory easily derived from the Strichartz estimates is the
component of uniqueness of the evolution (see Theorems~4.3.1 and 3.3.9
in \cite{Cazenave}).  

\begin{lemma}
Given the assumptions on $\beta$, there exist $\rho_1$, $\rho_2 \in [2,\infty)$ such that for each $0< M < \infty$, there exists $C(M) < \infty$ such that
\[ 
\| \beta( |u|^2) u - \beta(|v|)^2 v \|_{L^{\rho_1'}} \leq C(M) \| u-v \|_{L^{\rho_2}}
\]
for all $u,v \in H^1 (X)$ such that $\| u \|_{H^1}, \|v\|_{H^1} \leq M$.  As a result, if $u$ and $v$ are weak solutions of \eqref{e:nls} on a time interval $I$ with initial data $u_0 \in H^1$, then $u = v$.    
\end{lemma}

The lemma follows by applying the Strichartz estimates to the equation for $u-v$ to obtain an estimate
\begin{align}
\| u-v \|_{L^q L^r} \leq C \| u-v \|_{L^{q'} L^r} 
\end{align}
for $(q,r)$ an admissible Strichartz pair and the observation using H\"{o}lder's inequality that if for $0 \in I$, a finite interval,
\[
\| f \|_{L^b(I)} \leq C \| f \|_{L^a (I)}.
\]
with $b>a$, then $f = 0$ almost everywhere in $I$.

Let us now consider $\sigma = 0$.  Let us take 
\begin{equation}
\label{e:betasimple}
\beta (z^2) = g z^{\rho-1}, 
\end{equation} 
for $\rho >1$, $g \in \reals$.  Since we are in two dimensions, for  $1 < \rho < 3$, using Theorem
\ref{thm:strichartz} and \eqref{eq:inhomStrich}, we are able to apply
standard bootstrapping arguments to such nonlinear problems, which we
include here for completeness.  We take instead of
\eqref{e:nls} the corresponding integral equation\begin{eqnarray*}
S (u) = e^{-i t \Delta} u_0 + i \int_0^t e^{-i (t-s) \Delta} \beta ( |u|^2) u (s) ds,
\end{eqnarray*}
which we show serves as a contraction mapping $S: Y_T \to Y_T$, for 
\begin{eqnarray*}
Y_T = C ([-T,T], L^2 (X)) \cap L^p ([-T,T], L^q (X)
\end{eqnarray*}
with $p,q$ an allowed Strichartz pair.  Indeed, from \eqref{eq:inhomStrich}
and H\"older's inequality as the equation is by scaling invariance computed to be sub-$L^2$-critical there
exists $p_1$, $q_1$ Strichartz pairs such that  
\begin{eqnarray*}
\left\{ \begin{array}{c} 
p_1' \rho < p,  \  q_1' \rho = q, \\
\frac{2}{p} + \frac{2}{q} = 1, \ \frac{2}{p_1} + \frac{2}{q_1} = 1, \\
p,p_1 \in [2, \infty], \ q, q_1 \in [2,\infty)
\end{array} \right.
\end{eqnarray*}
and 
\begin{eqnarray*}
\| S(u) \|_{Y_T}  \leq  C( \| u_0 \|_{L^2} + T^\rho \| u \|_{Y_T}), 
\end{eqnarray*}
\begin{eqnarray*}
\| S(u) - S(v) \|_{Y_T}  \leq   \tilde C T^\lambda \| u-v \|_{Y_T} \| |u| + |v| \|_{Y_T}^{\rho-1}.
\end{eqnarray*}
See, for instance, the discussion in \cite{Cazenave}, Section~4.6 for a detailed example of how to apply the bootstrapping principle once such bounds of the solution map are proved.
Hence, for each $u_0 \in L^2$, there exists a time interval $T$ and an upper bound $U$ such that given $u,v \in B(0,U) \subset Y_T$, we have
\begin{eqnarray*} 
S : B(0,U) \to B(0,U), \\
\| S(u) - S(v) \|_{Y_T} \leq \frac12 \| u - v \|_{Y_T}.
\end{eqnarray*}

For $0 < \sigma < 1$, in $\reals^2$, interpolative results up the $H^s$ critical exponent $\rho = 1 + 2/(1-\sigma)$.  hold in Besov type spaces in which it is simple to take advantage of the Sobolev embeddings in $H^s$, see for instance Section~4.9 of \cite{Cazenave}.  Such spaces can be defined on $\cone$, but doing so goes beyond the scope of this note.  \end{proof}

\subsection{Global Nonlinear Results}

Let us take $\beta$ as in \eqref{e:betasimple} for simplicity and assume
Dirichlet boundary conditions.  A remaining open problem is to determine
how much of the above linear analysis can be extended to domains $X$ with
Neumann boundary conditions; this would permit us to address global
nonlinear questions in the Neumann case as well.
  
For the $L^2$ sub-critical case $1 < \rho < 3$, the natural $L^2$
conservation gives global well-posedness in $L^2$ for such equations
by simply iteration of the argument over this uniform time interval.  Up to
this point, our analysis has cared very little about the leading order sign
in the nonlinearity, which is generally irrelevant to finite time results.
For the $L^2$ critical/supercritical case ($3 \leq \rho < \infty$), however, one must
rely upon the natural $H^1$ conservation laws of such a system; such
results are known to hold provided the
nonlinearity is defocusing ($g > 0$)---see \cite{Cazenave}, Theorem~6.1.1, taking $\sigma = 1$.

For the defocusing case ($g>0$), global well-posedness in $H^1$ holds
immediately for any initial data for all powers of the nonlinearity since
the positive conserved energy in \eqref{e:energy} means that the $H^1$ norm
can bounded for all time at any scale.  For the focusing case ($g < 0$),
global well-posedness is a subtle phenomenon owing to the potential
existence of nonlinear bound states.  With sufficient regularity in weighted
spaces, a calculation by Weinstein \cite{Weinstein:1982} showed that in the
case $\rho = 3$ on $\reals^2$ there is a finite threshold of $L^2$ mass
below which a solution exists for all time.  The threshold is related to a
nonlinear bound state that gives an optimal constant for the
Gagliardo-Nirenberg inequality
\begin{eqnarray*}
\| f \|_{L^p (X)} \leq C_{opt} (X) \| u \|^{\alpha (p)}_{\dot H^1 (X)} \| u \|^{1-\alpha(p)}_{L^2 (X)},
\end{eqnarray*} 
where $p < \infty$ in two dimensions.\footnote{It would be of interest to see how
such a calculation translates to the settings of exterior polygonal domains
and in particular if the relevant Gagliardo-Nirenberg constant changes at
all on product cones or polygonal exterior domains.}  Following \cite{Cazenave}, Theorem $6.1.1$, if one can use the conservation of energy and mass to provide an a priori bound on the $H^1$ norm throughout the evolution, then the methods sketched
above yield the following theorem about Schr\"odinger evolution on $X,$
where we
employ the notation from Theorem~\ref{t:locwp}:
\begin{theorem}
\label{t:globwp}
Assume that there exists $0 < \epsilon < 1$, $M>0$ and $0<C(M)$ such that 
\begin{equation}
\left| \int_X B(u^2) d \Vol_X \right| \leq  (1- \epsilon) \| u \|_{H^1 (X)}^2 + C(M)
\end{equation}
for $u \in H^1 (X)$ such that $\| u \|_{L^2 (X)} \leq M$.  Then, given $\| u_0 \|_{L^2} \leq M$ in \eqref{e:nls}, we have $T_{min} = T_{max} = \infty$ in Theorem \ref{t:locwp}.
\end{theorem}

As we have \emph{global}-in-time Strichartz estimates, it is natural to
pursue the question of \emph{scattering} of solutions with general critical
or supercritical nonlinearities in two settings:
\begin{itemize}\item In the focusing case, 
for small enough
data in $H^1$ (possibly with the condition of finite variance or $\| x u
\|_{L^2 (X)}^2 < \infty$),
\item In the defocusing case, for general data.
\end{itemize}
Note that a scattering state, say $u_+$, can be easily seen to
depend upon global dispersive results in the sense that global existence
implies that in the $H^1$ norm we construct
  \begin{align*}
  u_+ &= \lim_{t \to \infty} e^{-i t \Delta} u (x,t) \\
  & = u_0 - i \int_0^\infty e^{-i s \Delta} \beta ( |u|^2)u (s) ds ,
  \end{align*}
  provided the integral is bounded.  See for instance Tao \cite{Tao:2006a},
  Chapter~3.6 for a discussion.  We consequently propose the following:
  \begin{conjecture}
  Given $\epsilon > 0$ sufficiently small and $u_0 \in H^1$, $\| u_0 \|_{H^1 (X)}  < \epsilon$, there exist $u_+$, $u_- \in H^1$ such that given $u$ a global solution to \eqref{e:nls}, we have
  \begin{equation*}
  \| u(t) - e^{ i t \lap} u_{\pm}  \|_{H^1} \to 0
  \end{equation*}
  as $t \to \pm \infty$ with
  \begin{equation*}
  M(u_+) = M(u_-) = M(u_0) \ \ \text{and} \ \ \| \nabla u_+ \|_{L^2 (X)}^2 = \| \nabla u_- \|_{L^2 (X)}^2 = E(u_0).
 \end{equation*}
 \end{conjecture}
 
 For more details,\footnote{Excellent
 introductory treatments of scattering are available in \cite{Cazenave},
 Sections~7.8 and 7.9 and \cite{SulemSulem}, Section~3.3.} we refer the reader to the treatment of scattering
 theory in higher dimensions of Strauss \cite{Strauss1}, \cite{Strauss2}
 and in particular to the treatment by Nakanishi in \cite{Nakanishi:1999},
 where the result was obtained in $\reals^2$.  See also \cite{CHVZ} for a
one-dimensional scattering result and \cite{Colliander:2009},
 \cite{PlanchonVega:2009} for two- and three-dimensional scattering results
 using Interaction Morawetz style estimates, which have potential for
 applying to exterior domains.  In two dimensions with a \emph{smooth}
 star-shaped obstacle, scattering is proved in 
 \cite{PlanchonVega:2011}; as the techniques there rely only upon integration by
 parts, they should extend to star-shaped polygonal domains (and possibly a
 broader class of non-trapping polygons) as well.  In the case of
 finite variance on $\RR^2$, this result follows from an extension of the
 pseudoconformal transformation, which is based upon a natural commuting
 vector field constructed from the Hamilton flow defined by the
 Schr\"odinger operator globally on $\reals^2$; thus we do not currently
 have a version of this conservation law for exterior domains $X$.
 
 Scattering can be shown in the defocussing case for any initial data given that the nonlinearity is critical/supercritical using a variant of the Morawetz estimates, which give local energy decay in the form of an estimate
 \begin{equation}
 \label{e:morawetz}
\int_X \frac{ |u|^p}{|x|^q} dx \leq \left( \frac{p}{2-q} \right)^q \| u \|_{L^p}^{p-q} \| \nabla u \|_{L^p}^q
\end{equation}
for $u \in W^{1,p} (X)$ where $q \leq 2$ and $0 \leq q \leq p$, $1 \leq p <
\infty$.  In recent works, such questions have been approached on
$\reals^2$ for less regularity using concentration compactness techniques
and interaction Morawetz estimates in the works of
\cite{Dodson:2010,Killip:2009}; one might hope to
generalize these results to product cones and exterior domains.

\bibliographystyle{abbrv}

\bibliography{papers}

\end{document}